\documentclass{article}
\usepackage{amsmath}
\usepackage{amssymb}
\usepackage{amsthm}
\usepackage{graphicx}
\usepackage{cite}

\newtheorem{theorem}{Theorem}[section]
\newtheorem{proposition}[theorem]{Proposition}
\newtheorem{corollary}[theorem]{Corollary}
\newtheorem{lemma}[theorem]{Lemma}
\theoremstyle{definition}
\newtheorem{definition}[theorem]{Definition}
\newtheorem{example}{Example}

\newtheorem{remark}[theorem]{Remark}

\newcommand{\oline}[1]{\mathbin{\overline{#1}}}
\newcommand{\uline}[1]{\mathbin{\underline{#1}}}

\def\otr{\oline{\triangleright}}
\def\utr{\uline{\triangleright}}

\date{}

\title{\Large \textbf{Partially Multiplicative Biquandles and Handlebody-Knots}}

\author{
Atsushi Ishii\footnote{Email: \texttt{aishii@math.tsukuba.ac.jp}. Partially supported by JSPS KAKENHI Grant Number 15K04866}
\and 
Sam Nelson\footnote{Email: \texttt{knots@esotericka.org}. Partially Supported by Simons Foundation Collaboration Grant 316709}
}

\begin{document}
\maketitle

\begin{abstract}  We introduce several algebraic structures related to
handlebody-knots, including \textit{$G$-families of biquandles}, 
\textit{partially multiplicative biquandles} and \textit{group decomposable
biquandles}. These structures can be used to color the semiarcs
in $Y$-oriented spatial trivalent graph diagrams representing
$S^1$-oriented handlebody-knots to obtain computable invariants 
for handlebody-knots and  handlebody-links. In the case of $G$-families of
biquandles, we enhance the counting invariant using the group $G$ to obtain
a polynomial invariant of handlebody-knots.
\end{abstract}

\bigskip

\quad
\parbox{5in}{

}

\section{Introduction}

Introduced in the early 1980s, \textit{quandles} are algebraic structures 
which can be used to distinguish knots and links by counting colorings of
\textit{arcs} (the portions going from one under-crossing to another) in 
an oriented knot or link diagram by elements of a fixed quandle
\cite{FR,J,M}. In \cite{FRS} and later \cite{KR}, quandles were generalized 
to \textit{biquandles} which can be used to distinguish oriented knots and 
links by counting colorings of the \textit{semiarcs} (the portions going 
from one under-crossing or over-crossing to another) in an oriented knot 
or link diagram.

In previous work such as \cite{I,IIJO,L}, quandles and related structures such
as \textit{$G$-families of quandles} and \textit{qualgebras} were used to
define invariants of spatial trivalent graphs and related structures 
such as handlebody-knots by 
coloring the arcs (now defined as portions going from one under-crossing or 
vertex to another). In this paper we generalize these structures to structures 
for coloring the semiarcs of handlebody-knot diagrams, now defined as
the portions of the diagram between under-crossing points, over-crossing 
points, and vertices.

The paper is organized as follows.
In Section \ref{B} we review the basics of biquandles and the counting 
invariant. In Section \ref{NPB} we introduce \textit{$n$-parallel biquandles}.
In Section \ref{GFB} we extend the notion of $G$-families of quandles to
the biquandle case. In Section \ref{MDB} we introduce partially multiplicative
biquandles and a special case, group decomposable biquandles. In Section 
\ref{Inv} we discuss invariants defined using these structures and provide
examples of their computation. We conclude in Section \ref{Q} with some 
questions for future research.

\section{Biquandles}\label{B}

We begin with a definition. (See \cite{EN} for more).

\begin{definition}
A \textit{biquandle} is a set $X$ with maps $\utr,\otr:X\times X\to X$
satisfying
\begin{itemize}
\item[(i)] For all $x\in X$, $x\utr x=x\otr x$,
\item[(ii)] For each $y\in X$, the maps $\alpha_y,\beta_y:X\to X$ and
$S:X\times X\to X\times X$ defined by
\[\alpha_y(x)=x\otr y,\quad \beta_y(x)=x\utr y \quad
\mathrm{and}\quad S(x,y)=(y\otr x,\ x\utr y)\] are bijective, and
\item[(iii)] The \textit{exchange laws} are satisfied:
\[\begin{array}{rcl}
(x\utr y)\utr (z\utr y) & = &(x\utr z)\utr (y\otr z) \\
(x\utr y)\otr (z\utr y) & = &(x\otr z)\utr (y\otr z) \\
(x\otr y)\otr (z\otr y) & = &(x\otr z)\otr (y\utr z).
\end{array}\]
\end{itemize}
A biquandle in which $x\otr y=x$ for all $x,y\in X$ is a \textit{quandle}.
\end{definition}

\begin{example}
For any set $X$ and bijection $\sigma:X\to X$, the 
operations $x\utr y=x\otr y=\sigma(x)$ define a biquandle called a 
\textit{constant action biquandle}.
\end{example}

\begin{example}
For any abelian group $A$ with automorphisms $s,t:A\to A$, the operations
\[x\utr y=t(x-y)+ s(y),\quad x\otr y=s(x)\]
define a biquandle called an \textit{Alexander biquandle}.
\end{example}

\begin{example}
For any group $G$, the operations $x\utr y=y^{-1}xy$ and $x\otr y=x$ define
a biquandle (indeed, a quandle) structure known as the \textit{conjugation
quandle} of $G$.
\end{example}

The biquandle axioms are motivated by the Reidemeister moves for oriented
knots and links. Specifically, if $X$ is a biquandle then an assignment of 
an element of $X$ to each semiarc in an oriented knot or link diagram
is a \textit{biquandle coloring} of the diagram if at every crossing we have
\[\includegraphics{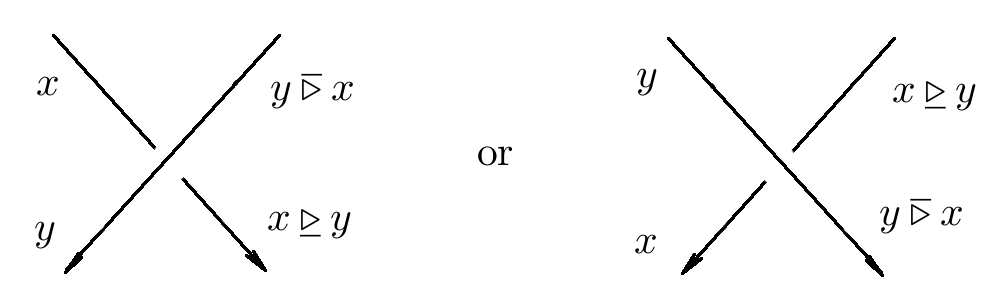}\]

Recall that two knot or link diagrams represent ambient isotopic knots or links
if and only if they differ by a sequence of \textit{Reidemeister moves}:
\[\includegraphics{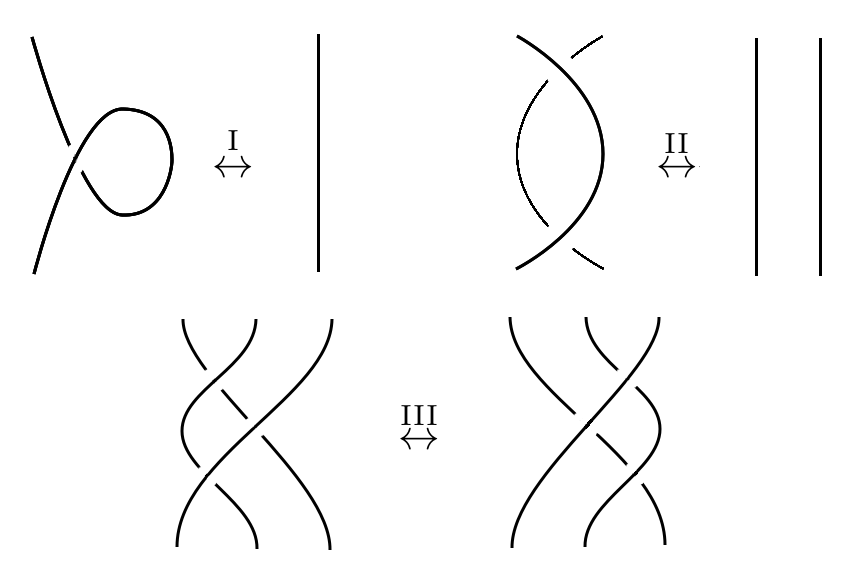}\]
 
It is then easy to check the following standard result (see also \cite{EN}).
\begin{theorem}
Let $D$ be an oriented knot or link diagram with a choice of biquandle 
coloring. Then for any Reidemeister move $\Omega$, there is a unique biquandle
coloring of the diagram $D'$ obtained from $D$ by applying $\Omega$ which 
agrees with the coloring on $D$ outside the neighborhood of the move.
\end{theorem}

\begin{definition}
Let $X$ be a biquandle and $D$ an oriented knot or link diagram. Then the
set of biquandle colorings of $D$ by $X$ is denoted $\mathcal{C}_X(D)$.
\end{definition}

Denote the cardinaility of $\mathcal{C}_X(D)$ by $|\mathcal{C}_X(D)|$. 
Then  we have the following:
\begin{corollary}
Let $X$ be a finite biquandle. Then for any two diagrams $D,D'$ of an oriented 
knot or link $L$, we have $|\mathcal{C}_X(D) |=|\mathcal{C}_X(D')|$.
\end{corollary}

\begin{definition}
For any biquandle $X$, the number $|\mathcal{C}_X(K)|$ of biquandle colorings 
of $K$ by $X$ is called the \textit{biquandle counting invariant} of $K$ with
respect to the biquandle $X$, denoted $\Phi^{\mathbb{Z}}_X(K)$.
\end{definition}

\begin{example}
Let $X=\mathbb{Z}_3=\mathbb{Z}/3\mathbb{Z}$ and set $t=1$ and $s=2$; then $X$ is an Alexander 
biquandle with operations $x\utr y=x+y$ and $x\otr y=2x$. We can compute
the biquandle counting invariant by row-reducing the matrices expressing 
the crossing relations over $\mathbb{Z}_3$. For example, the Hopf link
\[\includegraphics{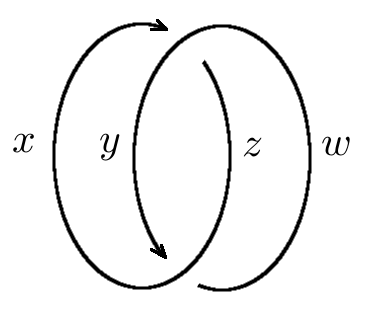}\]
has crossing equations $x+y=z,$ $2y=w$, $y+x=w$, $2x=z$ and thus $X$-coloring 
matrix which row-reduces over $\mathbb{Z}_3$ to
\[\left[\begin{array}{rrrr}
1 & 1 & 2 & 0 \\
0 & 2 & 0 & 2 \\
1 & 1 & 0 & 2 \\
2 & 0 & 2 & 0
\end{array}\right]
\longleftrightarrow
\left[\begin{array}{rrrr}
1 & 0 & 0 & 1 \\
0 & 1 & 0 & 1 \\
0 & 0 & 1 & 2 \\
0 & 0 & 0 & 0
\end{array}\right]
\]
so $|\mathcal{C}_X(\mathrm{Hopf\ Link})|=|X|=3$. The unlink of two
circles has 9 colorings, and hence the invariant detects the non-triviality 
of the  Hopf link.
\end{example}

\section{$n$-Parallel Biquandles}\label{NPB}

We would like to extend biquandles to algebraic structures suitable for 
defining counting invariants for spatial trivalent graphs and their quotient
structure, handlebody-knots. We will first develop some new notation.

\begin{definition}
Let $(X,\utr ,\otr )$ be a biquandle.
For $n>0$, we define
\begin{align*}
&a\utr^{[0]}b=a,
&&a\otr^{[0]}b=a, \\
&a\utr^{[n]}b=(a\utr^{[n-1]}b)\utr (b\utr^{[n-1]}b), \quad \mathrm{and}
&&a\otr^{[n]}b=(a\otr^{[n-1]}b)\otr (b\otr^{[n-1]}b).
\end{align*}
\end{definition}

\begin{example} Let $X$ be a biquandle. Then we have
\begin{align*}
&a\utr^{[1]}b=a\utr b,
&&a\utr^{[2]}b=(a\utr b)\utr (b\utr b),
&&a\utr^{[3]}b=((a\utr b)\utr (b\utr b))\utr ((b\utr b)\utr (b\utr b))
\end{align*}
et cetera.
\end{example}

\begin{remark}

If $(X,\underline{\triangleright},\overline{\triangleright})$
is a \textit{quandle}, i.e. a biquandle with $a\overline{\triangleright}b=a$
for all $a,b\in X$, then we have
\[a\underline{\triangleright}^{[n]}b
=a\underline{\triangleright}^nb
=\beta_b^n(a)\quad\mathrm{and}\quad
a\overline{\triangleright}^{[n]}b=a\]
where $\beta_y(x)=x\utr y$.
\end{remark}

Let $(X,\utr,\otr)$ be a biquandle. We will show via a series of lemmas that 
$(X,\utr^{[n]},\otr^{[n]})$ is also a biquandle.

\begin{lemma}\label{lem:3.3}
For $m,n\geq0$, we have
\begin{align*}
&(a\utr^{[m]}b)\utr^{[n]}(b\utr^{[m]}b)=a\utr^{[m+n]}b,
&&(a\otr^{[m]}b)\otr^{[n]}(b\otr^{[m]}b)=a\otr^{[m+n]}b.
\end{align*}
\end{lemma}

\begin{proof}
The proof is by induction on $n$.
When $n=0$, we have
\[(a\utr^{[m]}b)\utr^{[0]}(b\utr^{[m]}b)
=(a\utr^{[m]}b)=(a\utr^{[m+0]}b).\]
Now suppose that the equality holds for $n<k$.
Then we have
\begin{align*}
(a\utr^{[m]}b)\utr^{[k]}(b\utr^{[m]}b)
&=((a\utr^{[m]}b)\utr^{[k-1]}(b\utr^{[m]}b))
\utr ((b\utr^{[m]}b)\utr^{[k-1]}(b\utr^{[m]}b)) \\
&=(a\utr^{[m+k-1]}b)\utr (b\utr^{[m+k-1]}b)
=a\utr^{[m+k]}b
\end{align*}
as required.
In the same way, we can prove
$(a\otr^{[m]}b)\otr^{[n]}(b\otr^{[m]}b)=a\otr^{[m+n]}b$.
\end{proof}

We now verify that $(X,\utr^{[n]}, \otr^{[n]})$ satisfies the biquandle axioms.

\begin{lemma}\label{lem:3.4}
For $n\geq0$, we have $a\utr^{[n]}a=a\otr^{[n]}a$. 
\end{lemma}

\begin{proof}
Again, we proceed by induction on $n$.
When $n=0$, we have $a\utr^{[0]}a=a=a\otr^{[0]}a$.
Then suppose that the equality holds for $n<k$.
Then we have
\[ a\utr^{[k]}a
=(a\utr^{[k-1]}a)\utr (a\utr^{[k-1]}a)
=(a\otr^{[k-1]}a)\otr (a\otr^{[k-1]}a)
=a\otr^{[k]}a \]
as required.
\end{proof}

Next, we verify the exchange laws.

\begin{lemma}\label{lem:3.5}
For $m,n\geq0$, we have
\begin{align*}
(a\utr^{[m]}b)\utr^{[n]}(c\otr^{[m]}b)
&=(a\utr^{[n]}c)\utr^{[m]}(b\utr^{[n]}c), \\
(a\otr^{[m]}b)\utr^{[n]}(c\otr^{[m]}b)
&=(a\utr^{[n]}c)\otr^{[m]}(b\utr^{[n]}c), \\
(a\otr^{[m]}b)\otr^{[n]}(c\otr^{[m]}b)
&=(a\otr^{[n]}c)\otr^{[m]}(b\utr^{[n]}c).
\end{align*}
\end{lemma}

\begin{proof}
When $n=0$, we have
\begin{align*}
&(a\utr^{[m]}b)\utr^{[0]}(c\otr^{[m]}b)
=a\utr^{[m]}b=(a\utr^{[0]}c)\utr^{[m]}(b\utr^{[0]}c), \\
&(a\otr^{[m]}b)\utr^{[0]}(c\otr^{[m]}b)
=a\otr^{[m]}b=(a\utr^{[0]}c)\otr^{[m]}(b\utr^{[0]}c), \\
&(a\otr^{[m]}b)\otr^{[0]}(c\otr^{[m]}b)
=a\otr^{[m]}b=(a\otr^{[0]}c)\otr^{[m]}(b\utr^{[0]}c).
\end{align*}
When $m=0$, we have
\begin{align*}
&(a\utr^{[0]}b)\utr^{[n]}(c\otr^{[0]}b)
=a\utr^{[n]}c
=(a\utr^{[n]}c)\utr^{[0]}(b\utr^{[n]}c), \\
&(a\otr^{[0]}b)\utr^{[n]}(c\otr^{[0]}b)
=a\utr^{[n]}c
=(a\utr^{[n]}c)\otr^{[0]}(b\utr^{[n]}c), \\
&(a\otr^{[0]}b)\otr^{[n]}(c\otr^{[0]}b)
=a\otr^{[n]}c
=(a\otr^{[n]}c)\otr^{[0]}(b\utr^{[n]}c).
\end{align*}
Suppose that the three equalities hold for $n=1$, $m<k$.
Then we have
\begin{align*}
(a\utr^{[k]}b)\utr (c\otr^{[k]}b)
&=((a\utr^{[k-1]}b)\utr (b\utr^{[k-1]}b))
\utr ((c\otr^{[k-1]}b)\otr (b\otr^{[k-1]}b)) \\
&=((a\utr^{[k-1]}b)\utr (c\otr^{[k-1]}b))
\utr ((b\utr^{[k-1]}b)\utr (c\otr^{[k-1]}b)) \\
&=((a\utr c)\utr^{[k-1]}(b\utr c))
\utr ((b\utr c)\utr^{[k-1]}(b\utr c)) \\
&=(a\utr c)\utr^{[k]}(b\utr c).
\end{align*}
In the same way, we have
\begin{align*}
&(a\otr^{[m]}b)\utr (c\otr^{[m]}b)
=(a\utr c)\otr^{[m]}(b\utr c), \\
&(a\otr^{[m]}b)\otr (c\otr^{[m]}b)
=(a\otr c)\otr^{[m]}(b\utr c).
\end{align*}
Suppose that the three equalities hold for $n<k$.
Then we have
\begin{align*}
(a\utr^{[m]}b)\utr^{[k]}(c\otr^{[m]}b)
&=((a\utr^{[m]}b)\utr^{[k-1]}(c\otr^{[m]}b))
\utr ((c\otr^{[m]}b)\utr^{[k-1]}(c\otr^{[m]}b)) \\
&=((a\utr^{[k-1]}c)\utr^{[m]}(b\utr^{[k-1]}c))
\utr ((c\utr^{[k-1]}c)\otr^{[m]}(b\utr^{[k-1]}c)) \\
&=((a\utr^{[k-1]}c)\utr (c\utr^{[k-1]}c))
\utr^{[m]}((b\utr^{[k-1]}c)\utr (c\utr^{[k-1]}c)) \\
&=(a\utr^{[k]}c)\utr^{[m]}(b\utr^{[k]}c).
\end{align*}
In the same way, we have
\begin{align*}
&(a\otr^{[m]}b)\utr^{[n]}(c\otr^{[m]}b)
=(a\utr^{[n]}c)\otr^{[m]}(b\utr^{[n]}c) \quad \mathrm{and} \\
&(a\otr^{[m]}b)\otr^{[n]}(c\otr^{[m]}b)
=(a\otr^{[n]}c)\otr^{[m]}(b\utr^{[n]}c)
\end{align*}
as required.
\end{proof}

\begin{lemma}\label{lem:3.6}
Let $m,n\geq0$. Then
\begin{enumerate}
\item[(i)]
The maps $\utr^{[n]}a:X\to X$, $\otr^{[n]}a:X\to X$ are bijections for each $a\in X$.
\item[(ii)]
The map $S_{m,n}:X\times X\to X\times X;(x,y)\mapsto(y\otr^{[m]}x,x\utr^{[n]}y)$ is 
a bijection.
In particular, we have
\[a\underline{\triangleright}^{[n]}a=b\underline{\triangleright}^{[n]}b
\Leftrightarrow
a=b
\Leftrightarrow
a\overline{\triangleright}^{[n]}a=b\overline{\triangleright}^{[n]}b\]
\end{enumerate}
\end{lemma}

\begin{proof}
(i) From
$x\utr^{[n-1]}a=(x\utr^{[n]}a)\utr^{-1}(a\utr^{[n-1]}a)$,
the map $\utr^{[n]}a:X\to X$ is bijective for each $a\in X$.
In the same way, we see that the map $\otr^{[n]}a:X\to X$ 
is bijective for each $a\in X$.

(ii) Set
$a_{i,j}=(x\utr^{[i]}y)\utr^{[j]}(x\utr^{[i]}y)$ and
$a^{i,j}=(y\otr^{[j]}x)\otr^{[i]}(y\otr^{[j]}x)$.
Then
\[ S_{m,n}(a_{0,0},a^{0,0})=(a^{0,m},a_{n,0}). \]
Since $a_{i,j}=a_{i,0}\utr^{[j]}a_{i,0}$,
$a_{n,1},\ldots,a_{n,m-1}$ are uniquely determined from $a_{n,0}$, and
since $a^{i,j}=a^{0,j}\otr^{[i]}a^{0,j}$,
$a^{1,m},\ldots,a^{n-1,m}$ are uniquely determined from $a^{0,m}$.
Then
\begin{align*}
a_{i,j}\utr a^{i,j}
&=((x\utr^{[i]}y)\utr^{[j]}(x\utr^{[i]}y))
\utr^{[1]}((y\otr^{[j]}x)\utr^{[i]}(y\otr^{[j]}x)) \\
&=((x\utr^{[i]}y)\utr^{[j]}(x\utr^{[i]}y))
\utr^{[1]}((y\utr^{[i]}y)\otr^{[j]}(x\utr^{[i]}y)) \\
&=((x\utr^{[i]}y)\utr^{[1]}(y\utr^{[i]}y))
\utr^{[j]}((x\utr^{[i]}y)\utr^{[1]}(y\utr^{[i]}y)) \\
&=(x\utr^{[i+1]}y)\utr^{[j]}(x\utr^{[i+1]}y)
=a_{i+1,j}, \\
a^{i,j}\otr a_{i,j}
&=((y\otr^{[j]}x)\otr^{[i]}(y\otr^{[j]}x))
\otr^{[1]}((x\utr^{[i]}y)\otr^{[j]}(x\utr^{[i]}y)) \\
&=((y\otr^{[j]}x)\otr^{[i]}(y\otr^{[j]}x))
\otr^{[1]}((x\otr^{[j]}x)\utr^{[i]}(y\otr^{[j]}x)) \\
&=((y\otr^{[j]}x)\otr^{[1]}(x\otr^{[j]}x))
\otr^{[i]}((y\otr^{[j]}x)\otr^{[1]}(x\otr^{[j]}x)) \\
&=(y\otr^{[j+1]}x)\otr^{[i]}(y\otr^{[j+1]}x)
=a^{i,j+1}.
\end{align*}
Since
$S(a_{i,j},a^{i,j})
=(a^{i,j}\otr a_{i,j},a_{i,j}\utr a^{i,j})
=(a^{i,j+1},a_{i+1,j})$,
we have
$S^{-1}(a^{i,j+1},a_{i+1,j})=(a_{i,j},a^{i,j})$.
Moreover,
$a_{n-1,0}$ is uniquely determined from $a_{n,0},\ldots,a_{n,m-1}$, since we have
\begin{align*}
S^{-1}(a^{n-1,m},a_{n,m-1})&=(a_{n-1,m-1},a^{n-1,m-1}), \\
S^{-1}(a^{n-1,m-1},a_{n,m-2})&=(a_{n-1,m-2},a^{n-1,m-2}), \\
S^{-1}(a^{n-1,m-2},a_{n,m-3})&=(a_{n-1,m-3},a^{n-1,m-3}),\ldots \\
S^{-1}(a^{n-1,1},a_{n,0})&=(a_{n-1,0},a^{n-1,0})
\end{align*}
and 
$a^{0,m-1}$ is uniquely determined from $a^{0,m},\ldots,a^{n-1,m}$, since we have
\begin{align*}
S^{-1}(a^{n-1,m},a_{n,m-1})&=(a_{n-1,m-1},a^{n-1,m-1}), \\
S^{-1}(a^{n-2,m},a_{n-1,m-1})&=(a_{n-2,m-1},a^{n-2,m-1}), \\
S^{-1}(a^{n-3,m},a_{n-2,m-1})&=(a_{n-3,m-1},a^{n-3,m-1}),\ldots \\
S^{-1}(a^{0,m},a_{1,m-1})&=(a_{0,m-1},a^{0,m-1}).
\end{align*}
Repeating this, we see that $a_{0,0},a^{0,0}$ are uniquely determined from $a_{n,0},a^{0,m}$.
\end{proof}

\begin{definition}
Fix $n\geq0$.
We call $(X,\utr^{[n]},\otr^{[n]})$ the \textit{$n$-parallel biquandle} of 
$(X,\utr,\otr)$.
\end{definition}

\begin{example}
Consider the constant action biquandle $X=\mathbb{Z}_m$ with 
$x\otr y=x\utr y=x+1$. Then the $n$-parallel biquandle of $X$ is 
$\mathbb{Z}_m$ with $x\otr y=x\utr y=x+n$.
\end{example}

\begin{example}
More generally, if $\sigma:X\to X$ is a bijection then the $n$-parallel 
biquandle of the constant action biquandle $X$ with 
$x\utr y=x\otr y=\sigma(x)$
is $X$ with $x\utr^{[n]} y=x\otr^{[n]} y=\sigma^n(x)$.
\end{example}

\begin{proposition}
Let $X$ be an Alexander biquandle with operations
\[x\utr y=tx+(s-t)y\quad\mathrm{and}\quad x\otr y=sx.\]
Then the $n$-parallel biquandle of $X$ is the set $X$ with biquandle
operations
\[x\utr^{[n]} y=t^nx+(s^n-t^n)y\quad\mathrm{and}\quad x\otr^{[n]} y=s^nx.\]
\end{proposition}

\begin{proof}
As a base case, if $n=1$ we have
\[x\utr^{[1]} y=x\utr y=t^1x+(s^1-t^1)y
\quad\mathrm{and}\quad x\otr^{[1]} y=x\otr y=s^1x.\]

Now, suppose $x\utr^{[n-1]} y=t^{n-1}x+(s^{n-1}-t^{n-1}y)$; then we have
\begin{eqnarray*}
x\utr^{[n]} y & = & (x\utr^{[n-1]} y)\utr(y\utr^{[n-1]} y)\\
& = &  t(t^{n-1}x+(s^{n-1}-t^{n-1})y)+(s-t)(t^{n-1}y+(s^{n-1}-t^{n-1})y)\\
& = &  t^nx+[t(s^{n-1}-t^{n-1})+(s-t)t^{n-1}+(s-t)(s^{n-1}-t^{n-1})]y\\
& = &  t^nx+[ts^{n-1}-t^n+st^{n-1}-t^n+s^n-st^{n-1}-ts^{n-1}+t^n]y \\
& = &  t^nx+(s^n-t^n)y
\end{eqnarray*}
while $x\otr^{[n]} y=s^nx$ as required. 
\end{proof}

\section{$G$-Families of Biquandles}\label{GFB}

In this section we generalize a definition from \cite{IIJO} to the case of
biquandles.

\begin{definition}
Let $G$ be a group and $X$ a set. We say that
$(X,(\utr^g)_{g\in G},(\otr^g)_{g\in G})$ is a \textit{$G$-family of biquandles} if
\begin{itemize}
\item[(i)]
$a\utr^ga=a\otr^ga$
($\forall g\in G$, $\forall a\in X$)
\item[(ii)]
$\utr^ga:X\to X;x\mapsto x\utr^ga$ is a bijection ($\forall g\in G$, $\forall a\in X$)
\item[]
$\otr^ga:X\to X;x\mapsto x\otr^ga$ is a bijection ($\forall g\in G$, $\forall a\in X$)
\item[]
$S_{g,h}:X\times X\to X\times X;(x,y)\mapsto(y\otr^gx,x\utr^hy)$ is bijective ($\forall g,h\in G$)
\item[(iii)]
$(a\utr^gb)\utr^h(c\otr^gb)
=(a\utr^hc)\utr^{h^{-1}gh}(b\utr^hc)$
\item[]
$(a\otr^gb)\utr^h(c\otr^gb)
=(a\utr^hc)\otr^{h^{-1}gh}(b\utr^hc)$
\item[]
$(a\otr^gb)\otr^h(c\otr^gb)
=(a\otr^hc)\otr^{h^{-1}gh}(b\utr^hc)$
($\forall g,h\in G$, $\forall a,b,c\in X$)\quad and
\item[(iv)]
$a\utr^{gh}b=(a\utr^gb)\utr^h(b\utr^gb)$
\item[]
$a\otr^{gh}b=(a\otr^gb)\otr^h(b\otr^gb)$
($\forall g,h\in G$, $\forall a,b\in X$)
\end{itemize}
\end{definition}


\begin{definition}
Let $X$ be a biquandle. We define the \textit{idempotency index} and \textit{type} of $X$ by
\begin{itemize}
\item[]
$\operatorname{idem}X=\min\{n>0\,|\,a\utr^{[n]}a=a~(\forall a\in X)\}$\ and
\item[]
$\operatorname{type}X=\min\{n>0\,|\,a\utr^{[n]}b=a=a\otr^{[n]}b~(\forall a,b\in X)\}$.
\end{itemize}
\end{definition}

\begin{lemma}
Let $m\geq n\geq0$.
Let $X$ be a biquandle such that 
$\operatorname{idem}X,\operatorname{type}X<\infty$. Then we have
\begin{enumerate}
\item[(i)]
$\operatorname{idem}X\mid(m-n)$
$\Rightarrow$ $a\utr^{[m]}a=a\utr^{[n]}a$ and $a\otr^{[m]}a=a\otr^{[n]}a$ ($\forall a\in X$),
\item[(ii)]
$\operatorname{type}X\mid(m-n)$
$\Rightarrow$ $a\utr^{[m]}b=a\utr^{[n]}b$ and $a\otr^{[m]}b=a\otr^{[n]}b$ ($\forall a,b\in X$) and
\item[(iii)]
$\operatorname{idem}X\mid\operatorname{type}X$.
\end{enumerate}
\end{lemma}

\begin{proof}
\begin{enumerate}
\item[(i)] First, we compute 
\begin{align*}
a\otr^{[m]}a=a\utr^{[m]}a
&=a\utr^{[n+k\operatorname{idem}X]}a \\
&=(a\utr^{[\operatorname{idem}X]}a)
\utr^{[n+(k-1)\operatorname{idem}X]}
(a\utr^{[\operatorname{idem}X]}a) \\
&=a\utr^{[n+(k-1)\operatorname{idem}X]}a
=\cdots=a\utr^{[n]}a=a\otr^{[n]}a.
\end{align*}
\item[(ii)] Next, we have
\begin{align*}
a\utr^{[m]}b
&=a\utr^{[n+k\operatorname{type}X]}b \\
&=(a\utr^{[\operatorname{type}X]}b)
\utr^{[n+(k-1)\operatorname{type}X]}
(b\utr^{[\operatorname{type}X]}b) \\
&=a\utr^{[n+(k-1)\operatorname{type}X]}b
=\cdots=a\utr^{[n]}b.
\end{align*}
In the same way, we have
$a\otr^{[m]}b=a\otr^{[n]}b$.
\item[(iii)]
From (i), we have
\[ \{n\in\mathbb{Z}_{\geq0}\,|\,a\utr^{[n]}a=a~(\forall a\in X)\}
=(\operatorname{idem}X)\mathbb{Z}_{\geq0}. \]
It follows from $a\utr^{[\operatorname{type}X]}a=a$, that
$\operatorname{type}X\in(\operatorname{idem}X)\mathbb{Z}_{\geq0}$.
\end{enumerate}
\end{proof}

\begin{remark}
For $n\in\mathbb{Z}_{\operatorname{type}X}$ and
$a\utr^{[n]}b$, $a\otr^{[n]}b$ are well-defined.
\end{remark}

\begin{theorem}
Let $(X,\utr ,\otr )$ be a biquandle with $\operatorname{type}X<\infty$.
Set $G=\mathbb{Z}_{\operatorname{type}X}$.
Then $(X,(\utr^{[n]})_{n\in G},(\otr^{[n]})_{n\in G})$ is a 
$G$-family of biquandles called the \textit{$G$-family associated to $X$}.
\end{theorem}

\begin{proof}
This follows from Lemmas \ref{lem:3.3}, \ref{lem:3.4}, \ref{lem:3.5} and \ref{lem:3.6}.
\end{proof}

\begin{proposition}\label{prop:4.5} Let $X$ be a finite biquandle. Then
\[\operatorname{idem}X<\infty\quad \mathrm{and}\quad
\operatorname{type}X<\infty.\]
\end{proposition}

\begin{proof}
First, set
$n_a=\min\{n>0\,|\,a\utr^{[n]}a=a\}$. Then since 
$|\{a\utr^{[n]}a\,|\,n\in\mathbb{Z}_{\geq0}\}|<\infty$,
$\exists m,n>0$ ($m>n$) such that $a\utr^{[m]}a=a\utr^{[n]}a$.
Since
\[ (a\utr^{[m-n]}a)\utr^{[n]}(a\utr^{[m-n]}a)
=a\utr^{[m]}a=a\utr^{[n]}a, \]
we have $a\utr^{[m-n]}a=a$ and therefore $n_a\leq m-n<\infty$.
Since
\[ a\utr^{[kn_a]}a
=(a\utr^{[n_a]}a)\utr^{[(k-1)n_a]}(a\utr^{[n_a]}a)
=a\utr^{[(k-1)n_a]}a
=\cdots=a, \]
we have
$\operatorname{idem}X\leq\operatorname{lcm}(\{n_a\,|\,a\in X\})<\infty$
as required.

Next, set
$\underline{n}_b=\min\{n>0\,|\,(\utr^{[n_b]}b)^n=\mathrm{id}_X\}$ and
$\overline{n}_b=\min\{n>0\,|\,(\otr^{[n_b]}b)^n=\mathrm{id}_X\}$.
Since $X$ is finite, we have $\underline{n}_b,\overline{n}_b<\infty$.
Then
\begin{align*}
&a\utr^{[k\underline{n}_bn_b]}b
=(a\utr^{[n_b]}b)\utr^{[(k\underline{n}_b-1)n_b]}(b\utr^{[n_b]}b) \\
&=(a\utr^{[n_b]}b)\utr^{[(k\underline{n}_b-1)n_b]}b
=\cdots
=a\underbrace{\utr^{[n_b]}b\cdots\utr^{[n_b]}b}_{k\underline{n}_b}
=a.
\end{align*}
In the same way, we see $a\otr^{[k\overline{n}_bn_b]}b=a$.
Therefore
$\operatorname{type}X
\leq\operatorname{lcm}(\{\underline{n}_bn_b,\overline{n}_bn_b\,|\,b\in X\})
<\infty$ as required.
\end{proof}

\begin{example}\label{ex:gfam}
Let $X=\mathbb{Z}_5$ and set $t=2$ and $s=3$. Then we have Alexander biquandle
operations
\[\begin{array}{rclrcl}
x\utr^{[1]} y & = & 2x+(3-2)y = 2x+y, & x\otr^{[1]} y & = & 3x \\
x\utr^{[2]} y & = & 4x+(9-4)y =4x, & x\otr^{[2]} y & = & 9x=4x \\
x\utr^{[3]} y & = & 3x+(27-8)y =3x+4y, & x\otr^{[3]} y & = & 12x=2x \\
x\utr^{[4]} y & = &  x+(81-16)y =x, & x\otr^{[4]} y & = & 6x=x \\
\end{array}\]
so $(X,\utr,\otr)$ has type $4$; thus we have a $\mathbb{Z}_4$-family of
biquandles associated to $X$.
\end{example}

\begin{example}\label{ex:Z2}
Let $X=\mathbb{Z}_2$ with operations $x\utr y=x\otr y=x+1$. Then we have 
\[\begin{array}{rclrcl}
x\utr^{[1]} y & = & x+1  & x\otr^{[1]} y & = & x+1 \\
x\utr^{[2]} y & = & (x+1)+1=x, & x\otr^{[2]} y & = & (x+1)+1=x 
\end{array}\]
so $(X,\utr,\otr)$ has type $2$; thus we have a $\mathbb{Z}_2$-family of
biquandles associated to $X$.
\end{example}

\section{Partially Multiplicative Biquandles}\label{MDB}

In this section we will extend the idea of biquandle colorings to 
$S^1$-oriented handlebody-knots represented by $Y$-oriented spatial 
trivalent graph diagrams by adding a new operation at vertices.
Recall that the Reidemeister moves for handlebody-knots are given by
\[\includegraphics{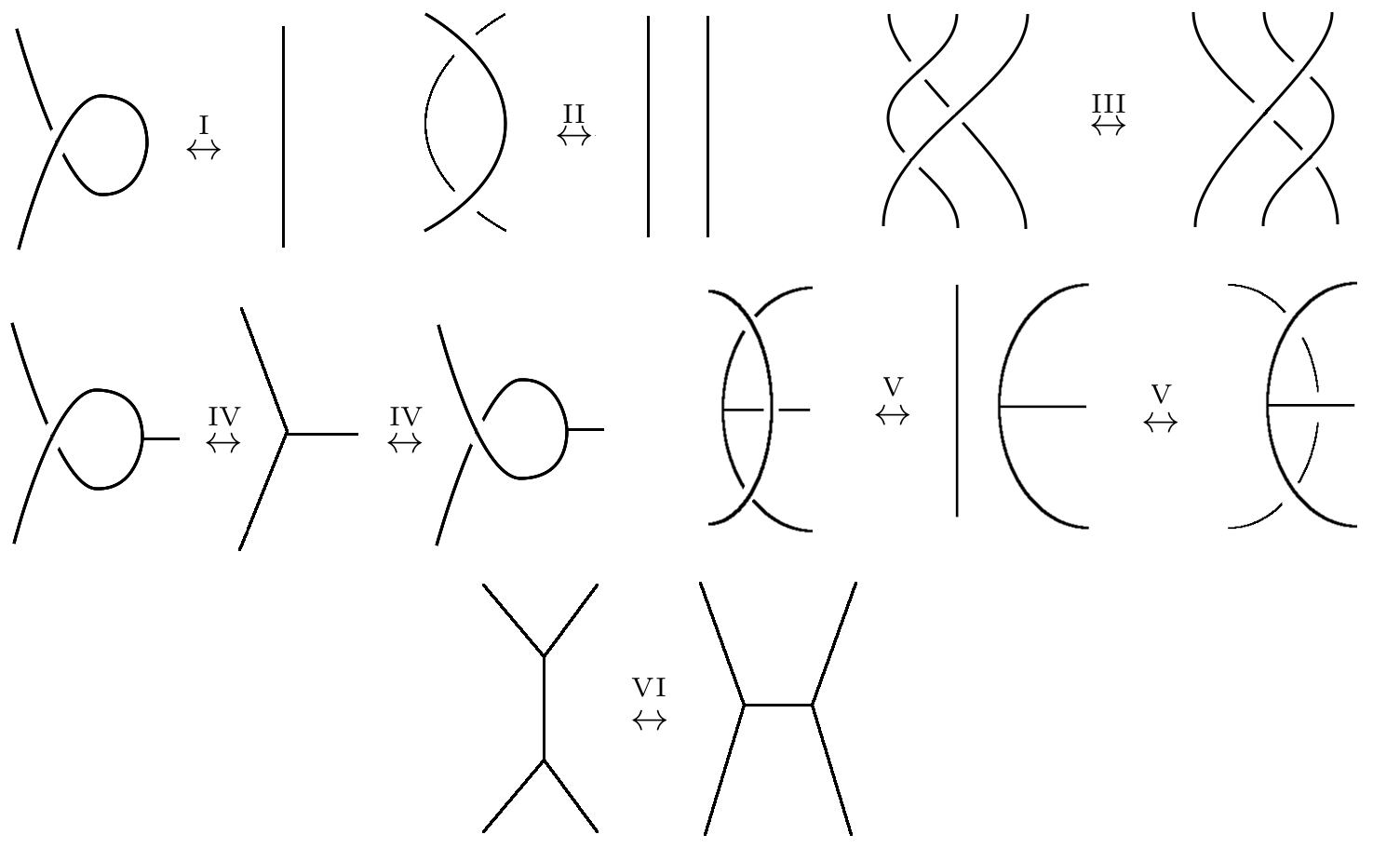}\]
and that spatial trivalent graph diagrams represent ambient isotopic 
spatial trivalent graphs if they are related by moves I, II, III, IV and V, 
while including the diagrammatic move VI (corresponding to the spatial IH move)
yields handlebody-knots. See \cite{I3,IIJO} for more.
A \textit{$Y$-orientation} of a spatial trivalent graph diagram
is a choice of direction for each edge in the underlying spatial graph 
such that no vertex is a source or a sink. An $S^1$-orientation of a 
handlebody-knot corresponds to a $Y$-orientation of a representative
spatial trivalent graph; see \cite{I3} for more.

\begin{definition} \label{def:PMB}
Let $(X,\utr ,\otr )$ be a biquandle, $D\subset X\times X$ a subset of the
Cartesian product of $X$ with itself, and $\cdot$ a map from $D$ to $X$
called a \textit{partial multiplication}. Then we say that
$(X,\utr ,\otr ,\cdot:D\to X;(a,b)\mapsto ab)$ is a \textit{partially 
multiplicative biquandle} if
\begin{itemize}
\item[(i)]
$x\mapsto ax$, $x\mapsto xb$ are injective,
\item[(ii)]
$(a,b\utr a)\in D\Leftrightarrow(b,a\otr b)\in D
\Rightarrow a(b\utr a)=b(a\otr b)$,
\item[(iii)]
$(a,b)\in D
\Leftrightarrow(a\utr x,b\utr (x\otr a))\in D
\Leftrightarrow(a\otr x,b\otr (x\utr a))\in D
\Rightarrow$
\begin{align*}
x\utr (ab)&=(x\utr a)\utr b, &
(ab)\utr x&=(a\utr x)(b\utr (x\otr a)), \\
x\otr (ab)&=(x\otr a)\otr b, &
(ab)\otr x&=(a\otr x)(b\otr (x\utr a)),
\end{align*}
\item[(iv)]
$(a,b),(ab,c)\in D\Leftrightarrow(b,c),(a,bc)\in D
\Rightarrow(ab)c=a(bc)$ and
\item[(v)]
$(a,b),(c,d)\in D$, $ab=cd$
$\Leftrightarrow$ $\exists e\in X$ such that $(a,e),(e,d)\in D$, $ae=c$, $ed=b$
\end{itemize}
If $D=\bigcup_{\lambda\in \Lambda} G_{\lambda}\times G_{\lambda}$ for a family of groups
$\{G_{\lambda}\ |\ \lambda\in \Lambda\}$ with group operation $(a,b)\mapsto ab$, 
then we say $(X,\utr,\otr,\cdot)$ is a \textit{group decomposable biquandle}.
\end{definition}

Partially multiplicative biquandles can be used to extend biquandle colorings
to handlebody-knots represented by spatial trivalent graphs
with \textit{$Y$-orientations}, i.e. directed
trivalent graphs in $\mathbb{R}^3$ without sources or sinks. Given such
a diagram $\Gamma$ and a partially multiplicative 
biquandle $X$, an assignment of elements of $X$ to the semiarcs of $\Gamma$
is an \textit{$X$-coloring} if at each crossing and vertex we have the
following:
\[\includegraphics{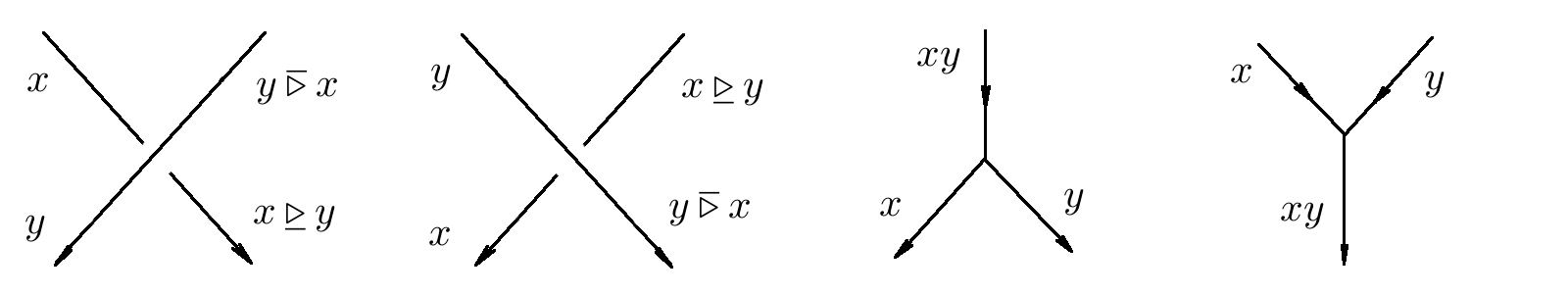}\]

We then have:

\begin{proposition}
Let $\Gamma$ be a $Y$-oriented spatial trivalent graph diagram with an 
$X$-coloring
by a partially multiplicative biquandle $X$. Then for any diagram $\Gamma'$
obtained from $\Gamma$ by a  handlebody-knot Reidemeister move, there is a unique $X$-coloring 
of $\Gamma'$ agreeing with the coloring on $\Gamma$ outside the neighborhood 
of the move.
\end{proposition}

\begin{proof}
This is a matter of checking the Reidemeister moves for $Y$-oriented spatial
trivalent graphs representing
handlebody-knots and comparing the axioms in definition \ref{def:PMB}. 
Invariance under moves I, II and III is well-known; see \cite{EN}, for 
instance. For each of the remaining moves, we illustrate with a choice of
Y-orientation; the other cases are similar.
\[\includegraphics{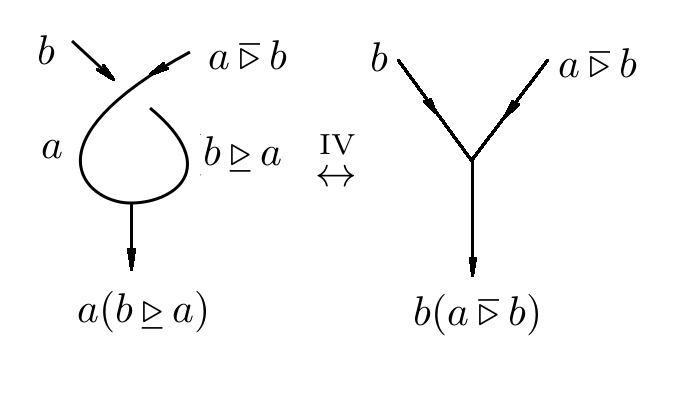}\ \raisebox{0.2in}{
\includegraphics{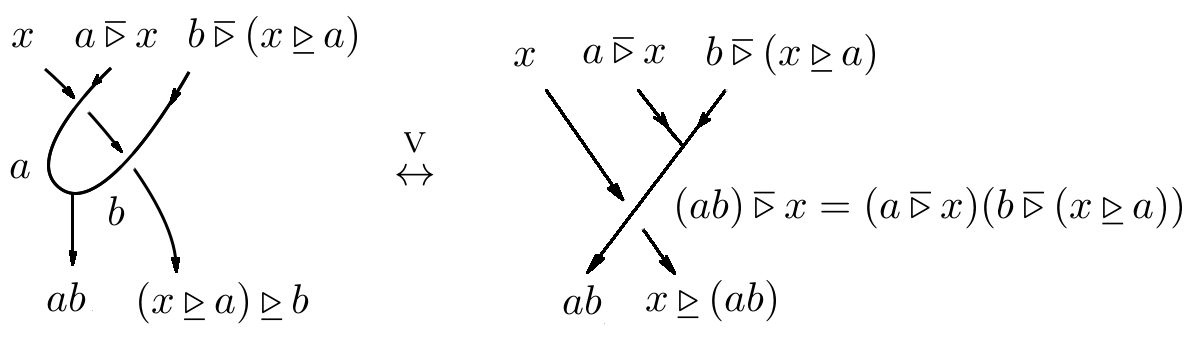}}\]
\[\includegraphics{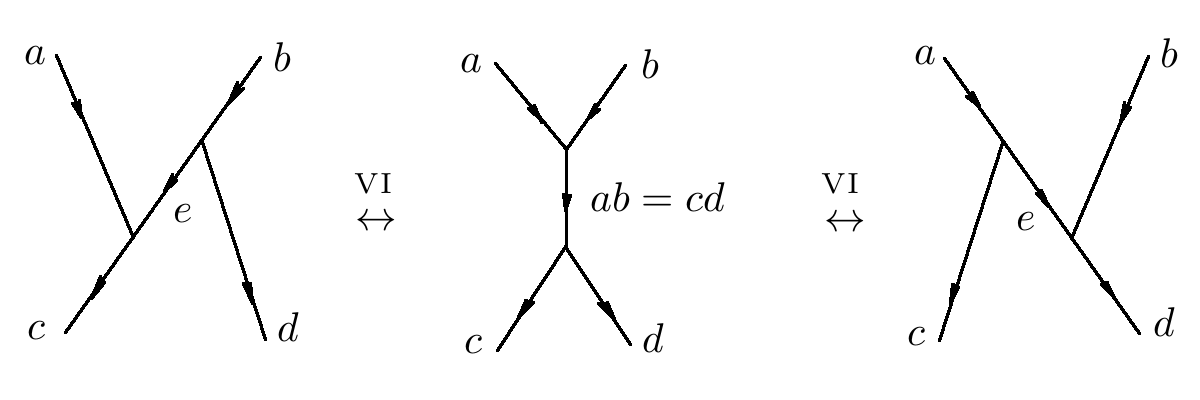}\quad \quad\includegraphics{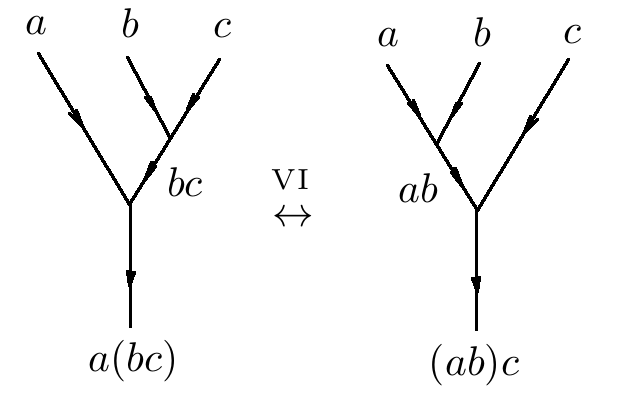}\]
\end{proof}

\begin{corollary}
The number of $X$-colorings of a $Y$-oriented spatial trivalent graph diagram
representing an $S^1$-oriented handlebody-knot by a partially multiplicative 
biquandle $X$ is invariant under the handlebody-knot Reidemeister moves.
\end{corollary}



We've already seen our main example of partially multiplicative biquandles:
$G$-families of biquandles.

\begin{proposition}
Let $(X,\utr^g,\otr^g)$ be a $G$-family of biquandles.
Set 
\[Q=X\times G\quad\mathrm{and}\quad 
D=\{((a,g),(a\utr^ga,h))\,|\,a\in X,g,h\in G\}\]
and define
\begin{align*}
(a,g)\utr (b,h)&=(a\utr^hb,h^{-1}gh),\\
(a,g)\otr (b,h)&=(a\otr^hb,g) \ \mathrm{and}\\
(a,g)\cdot(a\utr^ga,h)&=(a,gh).
\end{align*}
Then, $(Q,\utr ,\otr ,\cdot)$ is a partially multiplicative biquandle.
\end{proposition}

\begin{proof} This is a matter of verifying that the axioms of a partially 
multiplicative biquandle are satisfied. We must first show that 
$(Q,\utr ,\otr)$ is a biquandle. We have
\[(a,g)\utr (a,g)=(a\utr^ga,g^{-1}gg)
=(a\otr^ga,g)=(a,g)\otr (a,g)\]
so the first biquandle axiom is satisfied.
Since $\utr^ga,\otr^ga:X\to X$ is bijective,
$\utr (a,g):X\times G\to X\times G;(x,k)\mapsto(x\utr^ga,g^{-1}kg)$ and
$\otr (a,g):X\times G\to X\times G;(x,k)\mapsto(x\otr^ga,k)$ are bijective.
Since $S_{g,h}:X\times X\to X\times X;(x,y)\mapsto(y\otr^gx,x\utr^hy)$ is bijective,
$S:Q\times Q\to Q\times Q;((x,g),(y,h))\mapsto((y\otr^gx,h),(x\utr^hy,h^{-1}gh))$ is bijective,
and the second biquandle axiom is satisfied. Verifying the exchange laws, 
we have
\begin{align*}
((a,g)\utr (b,h))\utr ((c,k)\otr (b,h))
&=(a\utr^hb,h^{-1}gh)\utr (c\otr^hb,k) \\
&=((a\utr^hb)\utr^k(c\otr^hb),k^{-1}h^{-1}ghk) \\
&=((a\utr^kc)\utr^{k^{-1}hk}(b\utr^kc),k^{-1}h^{-1}ghk) \\
&=(a\utr^kc,k^{-1}gk)\utr (b\utr^kc,k^{-1}hk) \\
&=((a,g)\utr (c,k))\utr ((b,h)\utr (c,k)),
\end{align*}
\begin{align*}
((a,g)\otr (b,h))\utr ((c,k)\otr (b,h))
&=(a\otr^hb,g)\utr (c\otr^hb,k) \\
&=((a\otr^hb)\utr^k(c\otr^hb),k^{-1}gk) \\
&=((a\utr^kc)\otr^{k^{-1}hk}(b\utr^kc),k^{-1}gk) \\
&=(a\utr^kc,k^{-1}gk)\otr (b\utr^kc,k^{-1}hk) \\
&=((a,g)\utr (c,k))\otr ((b,h)\utr (c,k))
\end{align*}
and
\begin{align*}
((a,g)\otr (b,h))\otr ((c,k)\otr (b,h))
&=(a\otr^hb,g)\otr (c\otr^hb,k) \\
&=((a\otr^hb)\otr^k(c\otr^hb),g) \\
&=((a\otr^kc)\otr^{k^{-1}hk}(b\utr^kc),g) \\
&=(a\otr^kc,g)\otr (b\utr^kc,k^{-1}hk) \\
&=((a,g)\otr (c,k))\otr ((b,h)\utr (c,k)).
\end{align*}
Thus, $(Q,\utr ,\otr)$ is a biquandle.

Next, we note that
\[(a,g)\cdot\ :(a\utr^ga,h)\mapsto(a,g)(a\utr^ga,h)=(a,gh)\] 
and
\[\cdot\ (a\utr^ga,h):(a,g)\mapsto(a,g)(a\utr^ga,n)=(a,gh)\]
are injective.

Since
$a\utr^ga=b\utr^ga\Leftrightarrow a=b
\Leftrightarrow b\otr^hb=a\otr^hb$,
we have
\begin{align*}
&((a,g),(b,h)\utr (a,g))=((a,g),(b\utr^ga,g^{-1}hg))\in D \\
&\Leftrightarrow((b,h),(a,g)\otr (b,h))=((b,h),(a\otr^hb,g))\in D.
\end{align*}
Then
\begin{align*}
(a,g)((a,h)\utr (a,g)) & =  (a,g)(a\utr^ga,g^{-1}hg)=(a,hg) \\
&=  (a,h)(a\otr^ha,g)=(a,h)((a,g)\otr (a,h)).
\end{align*}

Since
$((a,g)\utr (x,k),(b,h)\utr ((x,k)\otr (a,g)))
=((a\utr^kx,k^{-1}gk),(b\utr^k(x\otr^ga),k^{-1}hk))$,
we have
\begin{align*}
((a,g)\utr (x,k),(b,h)\utr ((x,k)\otr (a,g)))\in D 
& \Leftrightarrow
(a\utr^kx)\utr^{k^{-1}gk}(a\utr^kx)
=b\utr^k(x\otr^ga) \\
&\Leftrightarrow
(a\utr^ga)\utr^k(x\otr^ga)
=b\utr^k(x\otr^ga) \\
&\Leftrightarrow a\utr^ga=b \\
& \Leftrightarrow((a,g),(b,h))\in D.
\end{align*}
Then
\begin{align*}
(x,k)\utr ((a,g)(a\utr^ga,h))
&=(x,k)\utr (a,gh) \\
&=(x\utr^{gh}a,h^{-1}g^{-1}kgh)
=((x\utr^ga)\utr^h(a\utr^ga),h^{-1}g^{-1}kgh) \\
&=(x\utr^ga,g^{-1}kg)\utr (a\utr^ga,h)
=((x,k)\utr (a,g))\utr (a\utr^ga,h) \\
\end{align*}
and
\begin{align*}
((a,g)(a\utr^ga,h))\utr (x,k)
&=(a,gh)\utr (x,k)
=(a\utr^kx,k^{-1}ghk) \\
&=(a\utr^kx,k^{-1}gk)((a\utr^kx)\utr^{k^{-1}gk}(a\utr^kx),k^{-1}hk) \\
&=(a\utr^kx,k^{-1}gk)((a\utr^ga)\utr^k(x\otr^ga),k^{-1}hk) \\
&=(a\utr^kx,k^{-1}gk)((a\utr^ga,h)\utr (x\otr^ga,k)) \\
&=((a,g)\utr (x,k))((a\utr^ga,h)\utr ((x,k)\otr (a,g))).
\end{align*}

Since
$((a,g)\otr (x,k),(b,h)\otr ((x,k)\utr (a,g)))
=((a\otr^kx,g),(b\otr^{g^{-1}kg}(x\utr^ga),h))$,
we have
\begin{align*}
((a,g)\otr (x,k),(b,h)\otr ((x,k)\utr (a,g)))\in D 
&\Leftrightarrow
(a\otr^kx)\utr^g(a\otr^kx)
=b\otr^{g^{-1}kg}(x\utr^ga) \\
&\Leftrightarrow
(a\utr^ga)\otr^{g^{-1}kg}(x\utr^ga)
=b\otr^{g^{-1}kg}(x\utr^ga) \\
&\Leftrightarrow a\utr^ga=b
\Leftrightarrow((a,g),(b,h))\in D.
\end{align*}
Then
\begin{align*}
(x,k)\otr ((a,g)(a\utr^ga,h))
&=(x,k)\otr (a,gh) \\
&=(x\otr^{gh}a,k)
=((x\otr^ga)\otr^h(a\utr^ga),k) \\
&=(x\otr^ga,k)\otr (a\utr^ga,h)
=((x,k)\otr (a,g))\otr (a\utr^ga,h), \\
\end{align*}
and
\begin{align*}
((a,g)(a\utr^ga,h))\otr (x,k)
&=(a,gh)\otr (x,k)
=(a\otr^kx,gh) \\
&=(a\otr^kx,g)((a\otr^kx)\utr^g(a\otr^kx),h) \\
&=(a\otr^kx,g)((a\utr^ga)\otr^{g^{-1}kg}(x\utr^ga),h) \\
&=(a\otr^kx,g)((a\utr^ga,h)\otr (x\utr^ga,g^{-1}kg)) \\
&=((a,g)\otr (x,k))((a\utr^ga,h)\otr ((x,k)\utr (a,g))).
\end{align*}

Since
$a\utr^{gh}a=(a\utr^ga)\utr^h(a\utr^ga)$,
we have
\begin{align*}
((a,g),(b,h)),((a,g)(b,h),(c,i))\in D 
&\Leftrightarrow a\utr^ga=b,a\utr^{gh}a=c \\
&\Leftrightarrow b\utr^hb=c,a\utr^ga=b \\
&\Leftrightarrow((b,h),(c,i)),((a,g),(b,h)(c,i))\in D.
\end{align*}
Then
\begin{align*}
&((a,g)(b,h))(c,i)=(a,gh)(c,i)=(a,ghi)=(a,g)(b,hi)=(a,g)((b,h)(c,i)).
\end{align*}

Finally, since
$(a\utr^ga)\utr^{g^{-1}i}(a\utr^ga)=a\utr^ia$,
there is an $(e,k)\in X\times G$ such that
\[((a,g),(e,k)),((e,k),(d,j))\in D, \quad (a,g)(e,k)=(c,i),
\quad\mathrm{and}\quad  (e,k)(d,j)=(b,h)\]
\begin{itemize}
\item[$\Leftrightarrow$]
$\exists e\in X$, $\exists k\in G$ such that $a\utr^ga=e$, $e\utr^ke=d$, $a=c$, $gk=i$, $e=b$, $kj=h$
\item[$\Leftrightarrow$]
$(a\utr^ga)\utr^{g^{-1}i}(a\utr^ga)=d$, $a=c$, $g^{-1}i=hj^{-1}$, $a\utr^ga=b$
\item[$\Leftrightarrow$]
$a\utr^ga=b$, $c\utr^ic=d$, $a=c$, $gh=ij$
\item[$\Leftrightarrow$]
$((a,g),(b,h)),((c,i),(d,j))\in D$, $(a,g)(b,h)=(c,i)(d,j)$
\end{itemize}
and $(Q,\utr ,\otr,\cdot)$ is a partially multiplicative biquandle.
\end{proof}

\begin{definition}
For any $G$-family of biquandles $X$, the partially multiplicative biquandle 
$(Q,\utr,\otr,\cdot)$ is the \textit{partially multiplicative biquandle
associated to $X$}.
\end{definition}

\begin{example}
Let $(X,(\utr^g)_{g\in\mathbb{Z}_4},(\otr^g)_{g\in\mathbb{Z}_4})$ be the 
$\mathbb{Z}_4$-family of biquandles from Example \ref{ex:gfam}, i.e.
$X=\mathbb{Z}_5$ with $t=2$ and $s=3$. Then we have
$Q=\mathbb{Z}_5\times\mathbb{Z}_4$ with operations
\[\begin{array}{rcl}
(a,g)\utr (b,h) & = & (2^ha+(3^h-2^h)b,g) \\
(a,g)\otr (b,h) & = & (3^ha,g) \\
(a,g)\cdot(a\utr^g a,h) & = & (a,g+h).
\end{array}\]
Then for instance, we have 
\[(2,3)\utr (3,3)=(2^3(2)+(3^3-2^3)3,3)=(3,3).\]
\end{example}

\section{Counting Invariants}\label{Inv}

Given a $Y$-oriented spatial trivalent graph diagram $\Gamma$ representing an 
$S^1$-oriented handlebody-knot and a partially multiplicative
biquandle $X$, the number of $X$-colorings of semiarcs in a diagram of 
$\Gamma$ is unchanged by Reidemeister moves by construction. Thus we have

\begin{definition}
Let $\Gamma$ be a $Y$-oriented spatial graph diagram representing an 
$S^1$-oriented handlebody-knot and $X$ a partially multiplicative biquandle. 
Then the \textit{partially multiplicative biquandle counting invariant} of
$\Gamma$ with respect to $X$ is the number of $X$-colorings of $\Gamma$, 
denoted $\Phi^{\mathbb{Z}}_X(\Gamma).$
\end{definition}

\begin{example} Let $X$ be the partially multiplicative biquandle associated to 
the Alexander biquandle $Z_5$ with $t=2$ and $s=3$ as in Example \ref{ex:gfam}.
Let us compute the number of $X$-colorings of the unknotted Theta graph 
$\Theta$ below:
\[\includegraphics{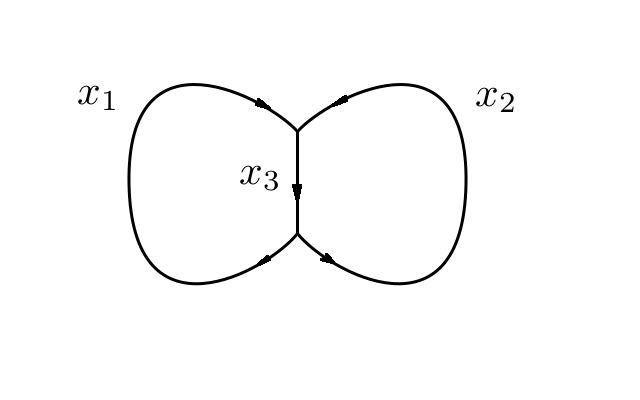}\]
For any $x_1=(a,g)\in\mathbb{Z}_5\times\mathbb{Z}_4$, we must have 
$x_2=(a\utr^g a,h)$ and $x_3=(a,gh)$. Then for any choice of 
$g,h\in\mathbb{Z}_4$ and $a\in\mathbb{Z}_5$, we get a valid $X$-coloring;
hence the counting invariant value is $\Phi_X^{\mathbb{Z}}(\Theta)=5(4)^2=80.$
\end{example}

When $X$ is a $G$-family of biquandles, we can take advantage of this extra
structure to enhance the counting invariant. Specifically, collecting 
together biquandle colorings which differ only in the the first component
gives us a way of filtering the set of $X$-colorings of our handlebody-knot 
diagram which is unchanged by Reidemeister moves, since forgetting the first
component yields a group coloring by $G$ (i.e., a group homomorphism from
$\pi_1(S^3\setminus \Gamma)$ to $G$). 
Given a $G$-coloring $\psi\in\mathrm{Hom}(\pi_1(S^3\setminus \Gamma),G)$, let 
us denote  by $\pi^{-1}(\psi)$ the set of $X$-colorings which project to 
$\psi$ by forgetting the first component $(a,g)\mapsto g$ on each semiarc. 
Then we have: 

\begin{definition}
Let $X$ be a $G$-family of biquandles and $\Gamma$ a $Y$-oriented spatial 
trivalent graph diagram representing an $S^1$-oriented
handlebody-knot. The \textit{$G$-enhanced biquandle counting invariant} of 
$\Gamma$ is the polynomial
\[\Phi_X^G(\Gamma)=\sum_{\psi\in\mathrm{Hom}(\pi_1(S^3\setminus \Gamma),G)} 
u^{|\pi^{-1}(\psi)|}\]
\end{definition}

\begin{proposition}
If two $Y$-oriented spatial trivalent graph diagrams $\Gamma$ and $\Gamma'$ 
representing $S^1$-oriented handlebody-knots are related by Reidemeister 
moves, then for any $G$-family of biquandles $X$, we have
\[\Phi_X^G(\Gamma)=\Phi_X^G(\Gamma').\]
\end{proposition}

\begin{example}
Let us illustrate the computation of the $G$-family enhanced counting invariant for the Kinoshita Theta graph $\Gamma$ below 
\[\includegraphics{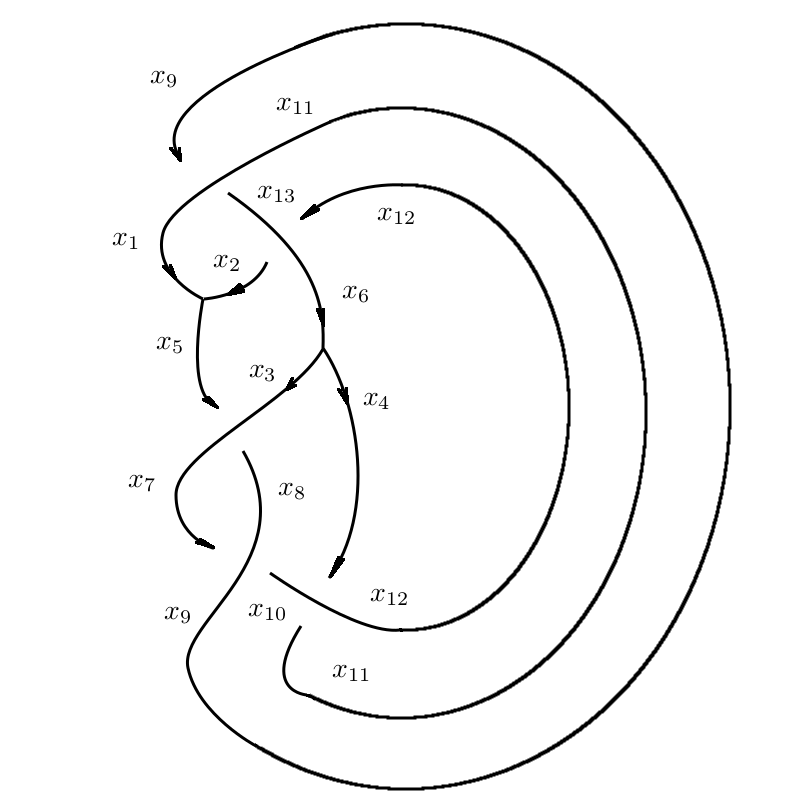}\]
with respect to the $G$-family of biquandles associated to
the Alexander biquandle $\mathbb{Z}_3$ with $t=1$ and $s=2$. We have 
\[\begin{array}{rclrcl}
x\utr^{[1]} y & = & 1x+(2-1)y = x+y, & x\otr^{[1]} y & = & 2x \\
x\utr^{[2]} y & = & 1^2x +(2^2-1^2)y=x & x\otr^{[2]} y & = & 2^2x=x \\
\end{array}\]
so $X$ has type 2 and we have a $\mathbb{Z}_2$-family of biquandles.
Then we have partially multiplicative biquandle operations
\[\begin{array}{rcl}
(a,g)\utr (b,h) & = & (a+(2^h-1)b,g) \\
(a,g)\otr (b,h) & = & (2^ha,g) \\
(a,g)\cdot(a\utr^g a,h) & = & (a,g+h).
\end{array}\]
Since the group colorings don't change at crossings with this abelian group
$G=\mathbb{Z}_2$, a choice of $g,h\in \mathbb{Z}_2$ determines the
group coloring. 
\[\includegraphics{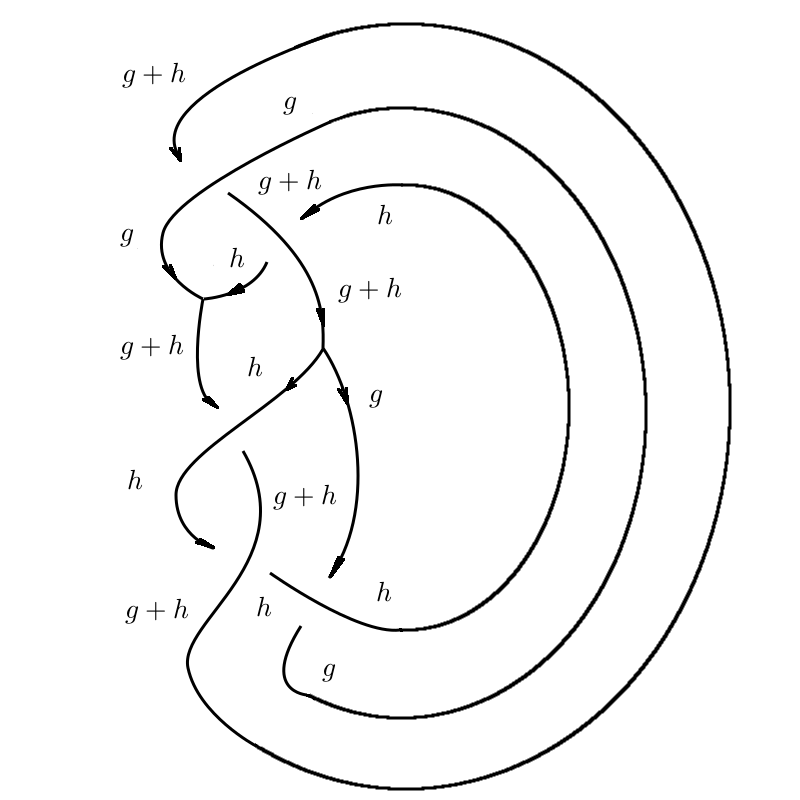}\]
For each assignment of $g,h\in\mathbb{Z}_2$ to 
the second components of $x_1=(a_1,g)$ and $x_2=(a_2,h)$
respectively, we get a system of linear equations over $\mathbb{Z}_3$. 
For instance, taking $g=0$ and $h=1$, we have
\[\begin{array}{rcl}
a_1 & = & a_5 \\
a_2 & = & a_1\utr^g a_1 \\
a_3 & = & a_6 \\
a_3 & = & a_7\otr^{g+h} a_5\\
a_4 & = & a_3\utr^h a_3\\
a_4 & = & a_{11}\utr^h a_{10}\\
a_6 & = & a_{13}\otr^{h} a_2 \\
a_8 & = & a_5\utr^{h} a_7\\
a_8 & = & a_9\otr^{h} a_7\\
a_{10} & = & a_7\utr^{g+h}a_9\\
a_{11} & = & a_1\otr^{g+h} a_9\\
a_{12} & = & a_2\utr^{g+h} a_{13}\\
a_{12} & = & a_{10}\otr^{g} a_{11}\\
a_{13} & = & a_9\utr^{g} a_1
\end{array}
\leftrightarrow
\left[\begin{array}{rrrrrrrrrrrrr}
2 & 0 & 0 & 0 & 1  & 0 & 0 & 0 & 0 & 0  & 0 & 0 & 0 \\
1 & 2 & 0 & 0 & 0  & 0 & 0 & 0 & 0 & 0  & 0 & 0 & 0 \\
0 & 0 & 2 & 0 & 0  & 1 & 0 & 0 & 0 & 0  & 0 & 0 & 0 \\
0 & 0 & 2 & 0 & 0  & 0 & 2 & 0 & 0 & 0  & 0 & 0 & 0 \\
0 & 0 & 2 & 2 & 0  & 0 & 0 & 0 & 0 & 0  & 0 & 0 & 0 \\
0 & 0 & 0 & 2 & 0  & 0 & 0 & 0 & 0 & 1  & 1 & 0 & 0 \\
0 & 0 & 0 & 0 & 0  & 2 & 0 & 0 & 0 & 0  & 0 & 0 & 2 \\
0 & 0 & 0 & 0 & 1  & 0 & 1 & 2 & 0 & 0  & 0 & 0 & 0 \\
0 & 0 & 0 & 0 & 0  & 0 & 0 & 2 & 2 & 0  & 0 & 0 & 0 \\
0 & 0 & 0 & 0 & 0  & 0 & 1 & 0 & 1 & 2  & 0 & 0 & 0 \\
2 & 0 & 0 & 0 & 0  & 0 & 0 & 0 & 0 & 0  & 2 & 0 & 0 \\
0 & 1 & 0 & 0 & 0  & 0 & 0 & 0 & 0 & 0  & 0 & 2 & 1 \\
0 & 0 & 0 & 0 & 0  & 0 & 0 & 0 & 0 & 1  & 0 & 2 & 0 \\
0 & 0 & 0 & 0 & 0  & 0 & 0 & 0 & 1 & 0  & 0 & 0 & 2 \\
\end{array}\right]
\]
which row-reduces over $\mathbb{Z}_3$ to
\[\left[\begin{array}{rrrrrrrrrrrrr}
1 & 2 & 0 & 0 & 0  & 0 & 0 & 0 & 0 & 0  & 0 & 0 & 0 \\
0 & 1 & 0 & 0 & 0  & 0 & 0 & 0 & 0 & 0  & 0 & 2 & 1 \\
0 & 0 & 1 & 0 & 0  & 2 & 0 & 0 & 0 & 0  & 0 & 0 & 0 \\
0 & 0 & 0 & 1 & 0  & 1 & 0 & 0 & 0 & 0  & 0 & 0 & 0 \\
0 & 0 & 0 & 0 & 1  & 0 & 0 & 0 & 0 & 0  & 0 & 2 & 1 \\
0 & 0 & 0 & 0 & 0  & 1 & 0 & 0 & 0 & 1  & 1 & 0 & 0 \\
0 & 0 & 0 & 0 & 0  & 0 & 1 & 0 & 1 & 2  & 0 & 0 & 0 \\
0 & 0 & 0 & 0 & 0  & 0 & 0 & 1 & 1 & 2  & 0 & 2 & 1 \\
0 & 0 & 0 & 0 & 0  & 0 & 0 & 0 & 1 & 0  & 0 & 0 & 0 \\
0 & 0 & 0 & 0 & 0  & 0 & 0 & 0 & 0 & 1  & 1 & 0 & 2 \\
0 & 0 & 0 & 0 & 0  & 0 & 0 & 0 & 0 & 0  & 1 & 1 & 2 \\
0 & 0 & 0 & 0 & 0  & 0 & 0 & 0 & 0 & 0  & 0 & 0 & 1 \\
0 & 0 & 0 & 0 & 0  & 0 & 0 & 0 & 0 & 0  & 0 & 0 & 0 \\
0 & 0 & 0 & 0 & 0  & 0 & 0 & 0 & 0 & 0  & 0 & 0 & 0 \\
\end{array}\right]\]
so we have $|\{0,1,2\}|=3$ $X$-colorings where $x_1=(a_1,0)$
and $x_2=(a_2,1)$. Repeating for $(g,h)=(1,1), (1,0)$ and
$(0,0)$, we obtain $3$ $X$-colorings for each, so we have
\[\Phi_X^G(\Gamma)=4u^3.\]
\end{example}

\section{Questions}\label{Q}

We end with some questions for future work.

What other enhancements of the $X$ counting invariant can be defined when
$X$ is a partially multiplicative biquandle or a group decomposable biquandle?
What about cocycle invariants in this setting?

\bibliography{ai-sn-rev}{}
\bibliographystyle{abbrv}

\noindent\textsc{Institute of Mathematics \\ 
University of Tsukuba \\
1-1-1 Tennodai \\
Tsukuba, Ibaraki 305-8571, Japan}

\

\noindent\textsc{Department of Mathematical Sciences\\
Claremont McKenna College \\
850 Columbia Ave. \\
Claremont, CA 91767, USA}

\end{document}